\documentclass{amsart}
\usepackage{epsfig}
\usepackage{graphics}
\usepackage{amsfonts}
\usepackage{amsmath}
\usepackage{amssymb,mathrsfs}
\usepackage{latexsym}

\newtheorem{theorem}{Theorem}[section]
\newtheorem{lemma}{Lemma}[section]
\newtheorem{proposition}{Proposition}[section]

\theoremstyle{definition}

\newtheorem{remark}{Remark}[section]

\numberwithin{equation}{section}

\begin{document}
\title{Fractional multiplicative processes}
\author{Julien Barral}
\address{INRIA Rocquencourt, B.P. 105, 78153 Le Chesnay Cedex, France}
\email{Julien.Barral@inria.fr}

\author{Beno\^\i t Mandelbrot}
\address{Yale University, Mathematics Department, New Haven CT 06520, USA}
\email{Benoit.Mandelbrot@yale.edu}

\begin{abstract}
Statistically self-similar measures on $[0,1]$ are 
limit of multiplicative cascades of random weights distributed on
the $b$-adic subintervals of $[0,1]$. These weights are i.i.d,
positive, and of expectation $1/b$. We extend these cascades
naturally by allowing the random weights to
take negative values. This yields martingales taking values in the space of continuous functions on $[0,1]$. Specifically, we
consider for each $H\in (0,1)$ the martingale $(B_{n})_{n\geq1}$
obtained when the weights take the values $-b^{-H}$ and $b^{-H}$, in order to get $B_n$ converging almost surely uniformly to a statistically self-similar function $B$
whose H\"{o}lder regularity and fractal properties are comparable with that of the fractional
Brownian motion of exponent $H$. This indeed holds when $H\in(1/2,1)$. Also the construction introduces a
new kind of law, one that it is stable under random
weighted averaging and satisfies the same functional equation as
the standard symmetric stable law of index $1/H$. When $H\in(0,1/2]$, to the
contrary, $B_n$ diverges almost surely. However, a natural normalization
factor $ a_n$ makes the normalized correlated random walk $
B_n / a_n$ converge in law, as $n$ tends to $\infty$, to the restriction to $[0,1]$ of the
standard Brownian motion. Limit theorems are also associated with
the case $H>1/2$.
\end{abstract}



\maketitle

\section{Introduction and results}
Measure-valued martingales associated with cascades were introduced in \cite{M1,M3} as a ``canonical'' model
for intermittent turbulence. They are generated by multiplicative cascades
of positive random weights distributed on the nodes of a homogeneous
tree. When non-degenerate, these martingales
converge to singular multifractal measures whose fine study has led to numerous developments, both in probability and geometric measure theories (see \cite{M1,KP,DL,Gu,K3,CK,HoWa, Fal,
Mol,Ar,B1,OW,BM,BS}). We consider the natural extension of these martingales consisting in
allowing the random weights to take negative values.

We simplify the exposition by using cascades in basis 2 (the necessary complements to extend our results in basis $b\ge 3$ are given in Remark~\ref{rem1.4}). The dyadic
closed subintervals of $[0,1]$ are naturally encoded by the nodes of
the binary tree $T=\bigcup_{n\ge 0}\{0,1\}^n$, with the convention
that $\{0,1\}^0$ contains the root of $T$ denoted $\emptyset$. As in
the definition of  positive canonical cascades \cite{M1}, we
associate to each element $w$ of $T$ a real valued random weight $W(w)$; these
weights are i.i.d and $\mathbb{E}(W)$ is defined and equal to $1/2$.
A sequence of random continuous piecewise linear functions $(B_{n})_{n\geq1}$
is then obtained as follows: $B_{n}(0)=0$; $B_{n}$ is linear over
every dyadic interval $I$ of the $n^{\mbox{{\small th}}}$
generation; if $I$ is encoded by the node $w_1w_2\cdots w_n$, i.e.
$I=I_w:=[\sum_{k=1}^nw_k2^{-k},2^{-n}+\sum_{k=1}^nw_k2^{-k}]$, the
increment of $B_n$ over $I$ is the product $W(w_1)W(w_1w_2)\cdots
W(w_1w_2\cdots w_n)$. If $W$ is non-negative, the derivatives in the distributions sense of the
functions $B_{n}$ form the
measure-valued martingale considered in \cite{M1,M3,KP}.

\smallskip

This paper investigates the signed cascades in
which the weight $W$ takes the same absolute value throughout, in order to generate fractional Brownian motion (fBm) like processes (see
\cite{Kol,M2} for the definition of fBm). It is
not difficult to see that in this case, for some $H\in(-\infty,1]$, $W$ must be of
the form $W= \epsilon\,  2^{-H}$, where $ \epsilon $ is a random
variable taking the values $1$ and $-1$ with respective
probabilities $ p^+=(1+2^{H-1})/{2}$ and $
p^-=(1-2^{H-1})/{2}$. Then let us reformulate the definition of  $(B_n)_{n\ge 1}$. 

Consider a sequence $(
\epsilon (w))_{w\in T}$ of independent copies of $ \epsilon $ and for every $n\ge 1$ and $w=w_1\cdots w_n\in\{0,1\}^n$ define
\begin{equation}\label{eps}
\boldsymbol{\epsilon}(w)=\prod_{k=1}^n \epsilon  (w_1\cdots w_k) \in\{-1,1\}.
\end{equation}
We can write $B_n$ as a normalized correlated random walk as follows: For $n\geq 1$ and
$0\leq k<2^{n}$ define $\xi_{k}^{(n)}=\boldsymbol{\epsilon}(w)$, where $w=w_1\cdots w_n$ is the
unique element of $\{0,1\}^n$ such that $t_{w}=\sum_{i=1}^nw_i2^{-i}=k2^{-n}$. The random
variables $\xi_{k}^{(n)}$, $0\le k<2^{n}$, are identically
distributed and they take values in $\{-1,1\}$. Also, consider the
random walk 
$$
 S_{r}^{(n)}=\sum_{k=0}^{r-1}\xi_{k}^{(n)},\ 0\le
r<2^n$$ 
(with the convention $S_{-1}^{(n)}=0$). Then for $t\in [0,1]$
we have
\begin{equation}\label{Bn}
B_n(t)=2^{-nH}\left[  S_{[2^{n}t]}^{(n)}+(2^{n}t-[2^{n}t])\xi_{\lbrack
2^{n}t]}^{(n)} \right].
\end{equation}
An equivalent definition of $(B_n)_{n\ge 1}$ is 
$$
 B_{n}(t)=2^{-nH}\int_{0}^{t} 2^{n}  \epsilon (u_1) \cdots \epsilon
(u_1\cdots u_n)\, \text{d}u,
$$
where the sequence $(u_k)_{k\ge 1}$ stands for the digits of $u$ in
basis $2$. This second definition shows by
inspection that this sequence of random continuous functions forms a
martingale with respect to the filtration $(\mathcal{F}_n)_{n\ge 1}$, where $\mathcal{F}_n=\sigma\big\{\epsilon (w):w\in\cup_{k=1}^n\{0,1\}^k\big \}$. 

For every $p\ge 0$ and $w=w_1\cdots w_p\in\{0,1\}^p$ we consider the copy of $( B_{n})_{n\ge 1}$ defined by
$$
  B_{n}(w) (t)=2^{-nH}\int_{0}^{t} 2^{n}  \epsilon (w\cdot u_1)
\cdots \epsilon  (w\cdot u_1\cdots u_n)\, \text{d}u,\quad (n\ge 1),
$$
where $w\cdot u_1\cdots u_k$ is the concatenation of the words $w$ and $u_1\cdots u_k$. By construction, $ B_{n}(\emptyset)= B_{n}$ and the following stochastic scaling invariance holds. With probability 1, for all $n\ge 1$ and $t\in I_w$
\begin{equation}\label{self-sim}
 B_{p+n}(t)- B_{p+n}(t_{w})= \boldsymbol{\epsilon}(w)2^{-pH}  B_{n}(w)\left( S_{w_{p}}^{-1}\circ
\cdots\circ S_{w_{1}}^{-1} (t)\right)
\end{equation}
where $ S_0(t)=t/2 \text{ and }S_1(t)=(t+1)/2$.

The previous properties of $ B_{n}$ may seem to suggest that if
$H\in(0,1)$, the construction provides a simple way to generate a sequence of normalized random walks (see (\ref{Bn})) converging almost
surely uniformly to a function $ B $ possessing scaling and fractal properties close to
those of a fBm of exponent $H$. In fact, our study of $(B_{n})_{n\ge 1}$ shows the situation to be subtler and heavily dependent on~$H$, a kind of phase transition arising at $H=1/2$. 

When $H\in (1/2,1)$, the martingale $(B_n)_{n\ge 1}$ indeed converges as expected as $n$ tends to $\infty$ (Theorem~\ref{th2}). This is illustrated in Figures~\ref{H=.95} and~\ref{H=.7}. The pointwise H\"{o}lder exponent of the almost sure limit $ B $ is equal to $H$ everywhere, and the Hausdorff dimension
of the graph of $ B $ is $2-H$. Moreover, the process $ B $
possesses scaling invariance properties relative to the dyadic grid,
with $H$ playing the role of a Hurst exponent, as can be seen by
letting $n$ tend to $\infty$ in~(\ref{self-sim}). Furthermore,  the normalized process $B/\sqrt{\mathbb{E}(B(1)^2)}$ converges in law to the standard Brownian motion as $H\searrow 1/2$ (Theorem~\ref{newTCL1}). Thus, $B$ shares a lot of properties with fBm of exponent $H$, though it has not stationary increments and it is not Gaussian (see Remark \ref{rem1.1}). When $H\in (-\infty,1/2]$, the martingale is not bounded in $L^2$ norm and it diverges. However, the normalized sequence $B_n/\sqrt{\mathbb{E}(B_n(1)^2)}$ converges in law to the standard Brownian motion as $n$ tends to $\infty$ (Theorem~\ref{newTCL}). This is illutrated in Figures~\ref{H=.5} and~\ref{H=-2}. When $H<1/2$ this result is a version of Donsker's theorem, but for triangular arrays with unusual strong correlations. When $H=1/2$, the same strong correlations hold, but $B_n/\sqrt{\mathbb{E}(B_n(1)^2)}$ corresponds to a correlated random walk normalized in the same unusual way as very different correlated random walks considered in \cite{Enr} and weakly converging to Brownian motion as well (see the discussion in Remark~\ref{rem1.3}). 

Our results are stated and commented in the following theorems and remarks. Then we relate them with some works on laws
that are stable under random weighted mean. 

$\mathcal{C}([0,1])$ will denote  the space of real-valued
continuous functions over $[0,1]$ endowed with the uniform norm
denoted by $\|\ \, \|_{\infty}$, and ${\rm Id}_{[0,1]}$ will denote the identity function over $[0,1]$. We refer to \cite{Falc} for the definitions of Hausdorff and box dimensions of sets in $\mathbb{R}^d$ as well as \cite{Bil} for the theory of the convergence of probability measures on metric spaces.

\medskip

\noindent
{\it The case $H\in (1/2,1]$.}

\begin{theorem}
\label{th2} Let $H\in(1/2,1]$. The $\mathcal{C}([0,1])$-valued martingale $( B_{n})_{n\geq1}$ converges almost surely and in $L^{q}$ norm for all $q\ge 1$ to a limit function of expectation ${\rm{Id}}_{[0,1]}$. Denote this limit by $ B $ and for all $w\in T$ the limit of $ B_{n}(w)$ by $ B (w)$.  With probability 1,

\begin{enumerate}

\item For all $p\ge 1$, $w\in \{0,1\}^p$ and $t\in I_w$
\begin{equation}\label{FONC3}
 B (t)- B (t_{w})= \boldsymbol{\epsilon}(w)\,  2^{-pH} B (w)\left( S_{w_{p}}^{-1}\circ
\cdots\circ S_{w_{1}}^{-1} (t)\right);
\end{equation}
\item $ B $ is $\alpha$-H\"older continuous for all $\alpha\in (0,H)$, and it has everywhere a pointwise H\"{o}lder exponent equal to $H$, i.e for all $t\in [0,1]$
$$
\displaystyle \liminf_{\substack{s\to t\\ s\neq t}}\frac{\log | B (s)- B (t)|}{\log |s-t|}=H;
$$
\item The Hausdorff and box dimensions of the graph of $ B $ are equal to $2-H$.
\end{enumerate}
\end{theorem}
 For $H\in (1/2,1)$ define $\sigma_H=(2-2^{2-2H})^{-1/2}=\sqrt{\mathbb{E}(B(1)^2)}$ (this equality will be justified in the proof of the next result) and denote $B$ by $B_H$.
\begin{theorem}\label{newTCL1} The family of continuous processes $\{B_H/\sigma_H\}_{H\in (1/2,1)}$ converges in law, as $H$ tends to $1/2$, to the restriction to $[0,1]$ of the standard Brownian motion.
\end{theorem}

\begin{remark}\label{rem1.1}
{\rm When $H=1$, the weights are positive and the construction coincides with the trivial positive cascade: with probability 1, $ B_{n}(t)=t$ for all $t\in [0,1]$ and $n\ge 1$. When $H\in (1/2,1)$, the limit process $ B-\text{Id}_{[0,1]} $ is not fractional Brownian motion. This can be seen on (\ref{FONC3}) since $\boldsymbol{\epsilon}(w)$ is not symmetric. Also, a
computation shows that the third moment of the centered random
variable $B (1)-1$ does not vanish, so the process is not Gaussian.}
\end{remark}

\medskip

\noindent
{\it The case $H\in [-\infty,1/2]$.}

\smallskip

For $H\in (-\infty,1/2]$, the sequence $(B_n)_{n\ge 1}$ is not bounded in $L^2$ norm. To get a natural normalization making it bounded in $L^2$ norm let
$$
\sigma=
\begin{cases}
\displaystyle\sqrt{1+(2^{2-2H}-2)^{-1}} & \mbox {if $H<1/2$}\\
\displaystyle1/{\sqrt{2}}& \mbox {if $H=1/2$}\\
\end{cases}
$$
and for $w\in T$ and $n\ge 1$ define
$$
X_{n}(w)=\begin{cases}
B_{n}(w)/\sigma 2^{n(1/2-H)}& \mbox {if $H<1/2$}\\
 B_{n}(w)/\sigma\sqrt{n}& \mbox {if $H=1/2$}\\
\end{cases}
$$
Also simply denote $X_{n}(\emptyset)$ by $X_{n}$. The process $X_n$ is equivalent to $B_n/\sqrt{\mathbb{E}(B_n(1)^2)}$ as $n$ tends to $\infty$ (this fact will be justified in the proof of the next result). If we let $H$ tend to $-\infty$ in the definition of $\epsilon$ and $\sigma$, then $\epsilon$ becomes a symmetric random variable taking values in $\{-1,1\}$, $\sigma=1$, and the sequence $(X_n)_{n\ge 1}$ has the natural extension to the case $H=-\infty$ given by $X_n(t)=\displaystyle\frac{1}{\sqrt{2^{n}}}\left[  S_{[2^{n}t]}^{(n)}+(2^{n}t-[2^{n}t])\xi_{\lbrack
2^{n}t]}^{(n)} \right]$ (see Remark~\ref{rem1.3}).  
\begin{theorem}\label{newTCL}
For every $H\in [-\infty,1/2]$ the sequence of continuous processes $(X_{n})_{n\ge 1}$ converges in law, as $n$ tends to $\infty$, to the restriction to $[0,1]$ of the standard Brownian motion.
\end{theorem}

\begin{remark}\label{rem1.2}
\label{dege} {\rm When $H\in (-\infty,1/2)$, 
$\limsup_{n\rightarrow\infty}\Vert
B_{n}\Vert_{\infty} {2^{-n(1/2-H)}} >0$ almost surely by Theorem~\ref{newTCL}. Thus the martingale $(
B_{n})_{n\ge 1}$ diverges in $\mathcal{C}([0,1])$. The same property
holds when $H=1/2$. Besides, Theorem~\ref{th2} says that $(
B_{n})_{n\ge 1}$ converges almost surely uniformly to a limit of expectation $\text{Id}_{[0,1]}$ when $H>1/2$. Consequently, the convergence properties of non-positive canonical
cascades strongly depend on the random weight used to generate the process. This contrasts with the positive canonical cascades martingales, which
always converge almost surely uniformly (either to a non-trivial
limit with expectation $\text{Id}_{[0,1]}$, or to \rm{0}, see
\cite{M1,KP}). }
\end{remark}

\begin{remark}\label{rem1.3}
{\rm When $H\in (-\infty,1/2]$, due to (\ref{Bn}) we have 
\begin{equation}\label{XN}
X_{n}(t)=
\begin{cases}
\displaystyle\frac{1}{\sigma\sqrt{2^{n}}}\left[  S_{[2^{n}t]}^{(n)}+(2^{n}t-[2^{n}t])\xi_{\lbrack
2^{n}t]}^{(n)} \right] &\text{if }H<1/2\\
 \displaystyle\frac{1}{\sigma\sqrt{n2^{n}}}\left[  S_{[2^{n}t]}^{(n)}+(2^{n}t-[2^{n}t])\xi_{\lbrack
2^{n}t]}^{(n)} \right] &\text{if }H=1/2
\end{cases}.
\end{equation}
When $H<1/2$, the form of $X_n$ is familiar from Donsker's theorem
(see \cite{Bil}) and its extensions to triangular arrays of 
random variables that are weakly dependent 
(see \cite{Bil,Douk}). However, the correlations of the $X_n$ dyadic increments are closely related to the natural ultrametric distance on $T$ and it seems difficult to find a way to reduce the behavior of $(X_n)_{n\ge 1}$ to that of random walks with weakly dependent increments. When $H=1/2$, the $X_n$ dyadic increments are correlated as well, and the normalization of the random walk is similar to the unusual one met in the proof of Theorem 2 in \cite{Enr} (see also Lemma 5.1 of \cite{T}) to obtain the weak convergence to Brownian motion of certain centered stationary Gaussian random walks. 

If we denote $X_n(w)(1)$ by $Y_n(w)$, the relation (\ref{fonc2}) below
yields
\begin{equation}\label{foncnorm}
Y_{n+1}=\begin{cases}
\displaystyle \frac{\epsilon (0)}{\sqrt{2}}Y_n(0)+ \frac{\epsilon (1)}{\sqrt{2}}Y_n(1)&\text{if } H<1/2\\
\displaystyle \sqrt{\frac{n}{n+1}}\Big (\frac{\epsilon
(0)}{\sqrt{2}}Y_n(0)+ \frac{\epsilon (1)}{\sqrt{2}}Y_n(1)\Big
)&\text{if }H=1/2 .
\end{cases}
\end{equation}
Consequently, assuming that $X_n$ converges in law, we
can guess thanks to (\ref{foncnorm}) that the weak limit of $Y_n$ must be the standard normal distribution. Actually, to identify this limit we exploit the recursive equations (\ref{foncnorm}) as
well as recursive equations satisfied by the moments of  the
standard normal distribution (see (\ref{momnorm}) in the proof of
Lemma~\ref{ident}). A similar approach exploiting the functional equation (\ref{FONC}) is used to prove Theorem~\ref{newTCL1}.

Letting $H$ tend to $-\infty$ yields $\sigma=1$ and a random variable $\epsilon$ that takes the
values $-1$ and $1$ with equal probability $1/2$ so that the random
walk $ S_{r}^{(n)}$ becomes symmetric. In this case, the convergence
in law to Brownian motion of $X_n$ (defined as in (\ref{XN}) in the limit $H=-\infty$) follows from standard arguments,
since $X_n$ conditioned with respect to
$\mathcal{G}_{n-1}=\sigma\big \{\boldsymbol{\epsilon} (w):w\in\{0,1\}^{n-1}\big \}$ satisfies
the Donsker's theorem assumptions (given $\mathcal{G}_{n-1}$, the $\xi_{k}^{(n)}$s are symmetric, independent, and take values $-1$ and $1$). 

If $H\in (1/2,1)$ and $\sigma$ is defined as $\sigma=\sqrt{\mathbb{E}( B (1)^{2})-1}$, the same kind of argument can be used to prove that $X_n=(B-B_n)/\sigma 2^{n(1/2-H)}$ also converges in law to Brownian motion. Indeed, due to (\ref{FONC3}), conditionally on
$\sigma\big \{\boldsymbol{\epsilon} (w):w\in\{0,1\}^{n}\big \}$, the increments of the process $2^{n/2}X_n$ over the dyadic
intervals of generation $n$ are $2^n$ independent centered random variables distributed like $(B(1)-1)/\sigma$ or $-(B(1)-1)/\sigma$, namely the 
$\boldsymbol{\epsilon} (w)(B(w)(1)-1)/\sigma$, $w\in \{0,1\}^n$, whose standard deviation is
equal to 1.}
\end{remark}
\subsection*{A link with laws that are stable under random weighted mean}

For $n\geq0$ and $w\in T$ we denote by $Z_{n}(w)$ the random
variable $ B_{n}(w)(1)$, with the convention $B_{0}(w)(1)=1$. We simply write $Z_{n}$ for $Z_{n}(\emptyset)$. By construction, for every $n\geq1$
\begin{equation}
\label{fonc2}Z_{n}=2^{-H}\epsilon (0)Z_{n-1}(0)+2^{-H}\epsilon (1)Z_{n-1}(1),
\end{equation}
where the random variables $\epsilon(0)$, $\epsilon (1)$, $Z_{n-1}
(0)$ and $Z_{n-1}(1)$ are mutually independent, $\epsilon(0)$ and
$\epsilon(1)$ are copies of $ \epsilon $, and $Z_{n-1}(0)$ and
$Z_{n-1}(1)$ are copies of $Z_{n-1}$. Relation (\ref{fonc2}) is
central in the sequel. When the martingale $(Z_n)_{n\ge 1}$ does
converge to a non trivial limit $Z$ (see Theorem~\ref{th2}), it
follows from (\ref{fonc2}) that the probability distribution of $Z$
provides a new family of what has
been called law stable by random weighted mean or fixed points of
the smoothing transformation (\cite{M1,DL,Gu}). Indeed, there exist
two independent copies $Z(0)$ and $Z(1)$ of $Z$, and two independent
and identically distributed random variables $W(0)$ and $W(1)$ ---
namely, $2^{-H}\epsilon (0)$ and $2^{-H}\epsilon (1)$ --- such that
$(W(0),W(1))$ is independent of $(Z(0),Z(1))$ and $Z$ satisfies the
following equality in distribution ($\equiv$)
\begin{equation}\label{foncgen}
Z\equiv W(0)Z(0)+W(1)Z(1).
\end{equation}
When $(W(0),W(1))$ is positive, the non-trivial positive solutions
of this equation are described in \cite{M1,KP,DL,Gu}. A class of non-positive
solutions of (\ref{foncgen}) with positive $(W(0),W(1))$ has been
exhibited in \cite{Liu3}; it naturally includes classical symmetric
stable laws of index $\alpha\in [1,2]$, which obey (\ref{foncgen})
when $W(0)=W(1)=2^{-H}$ with $H=1/\alpha\in [1/2,1]$. Actually, the
classical symmetric stable law of index $\alpha=1/H\in [1,2]$
satisfies equation (\ref{foncgen}) under the form $Z\equiv 2^{-H}\eta
(0)Z(0)+2^{-H}\eta (1)Z(1)$ as soon as $\eta(0)$ and $\eta (1)$ are independent, take values
$-1$ and $1$, and are independent of $(Z(0),Z(1))$, whatever be the distributions of  $\eta(0)$ and
$\eta(1)$. Consequently, when $(\eta(0),\eta(1))=(\epsilon(0),\epsilon(1))$, Theorem~\ref{th2} provides for each $H\in
(1/2,1]$ another probability distribution obeying the same functional
equation as the classical symmetric stable law of index $1/H$.
It is worth noting that the statistically self-similar stochastic processes
associated with these solutions have very different behaviors. In
the first case, if $H=1/\alpha\in (1/2,1]$ the process is a
symmetric stable L\'evy process $L_\alpha$ of index $\alpha$ (see \cite{Be}), so the
distributions of the increments have no finite moments of order
larger than or equal to $\alpha$, and the sample path of $L_\alpha$ have a dense
set of discontinuities and are multifractal \cite{Jaf}. In the
second case, the process is the random function $B$ of
Theorem~\ref{th2}, the distributions of the dyadic increments have a
finite moment of order $p$ for all $p>0$, and the sample path of $B$
are continuous and monofractal.

\begin{remark}\label{rem1.4}
Both the construction and results  extend to the case when the
construction grid is  $b$-adic  with $b\ge 3$. Then $W= \epsilon
b^{-H}$, where $ \epsilon =1$ with probability $(1+b^{H-1})/2$ and $
\epsilon =-1$ with probability $(1-b^{H-1})/2$. The same results
hold after formal replacement of the basis 2 by the basis $b$. Also,
$\sigma=\sqrt{1-1/b} $ if $H=1/2$,
$\sigma=\sqrt{1+(b-1)/(b^{2-2H}-b)} $ if $H<1/2$, and $\sigma_H=\sqrt{b-1}/\sqrt{b-b^{2-2H}}$ if $H>1/2$.
\end{remark}

Theorems~\ref{th2}, \ref{newTCL} and \ref{newTCL1} are proved in Sections~\ref{sec2}, \ref{sec3} and \ref{sec4} respectively.

\begin{center}
\begin{figure}
\begin{center}
{\includegraphics[width=10cm,height = 10cm]{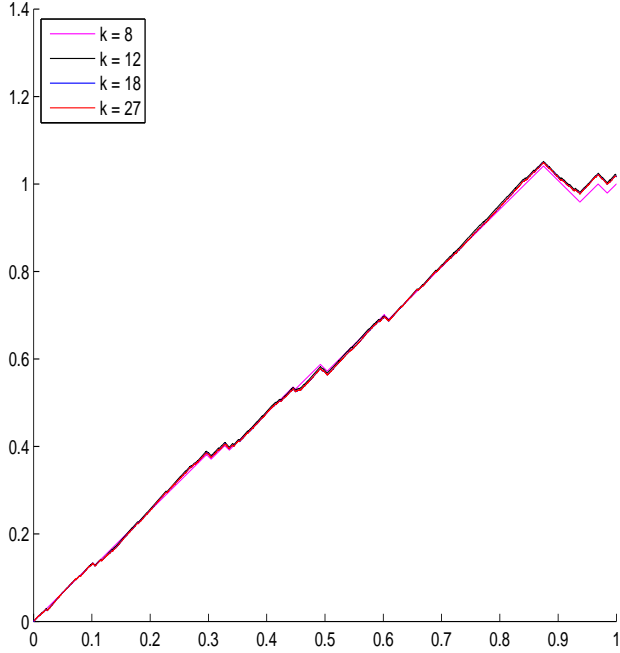}}
\vskip -1cm
\caption{$B_k$ for $k=8,\ 12,\ 18,\ 27$ in the case $b=2$ and $H=0.95$: Fast strong convergence.}
\label{H=.95}
\end{center}
\end{figure}
\end{center}
\begin{center}
\begin{figure}
\begin{center}

{\includegraphics[width=10cm,height = 10cm]{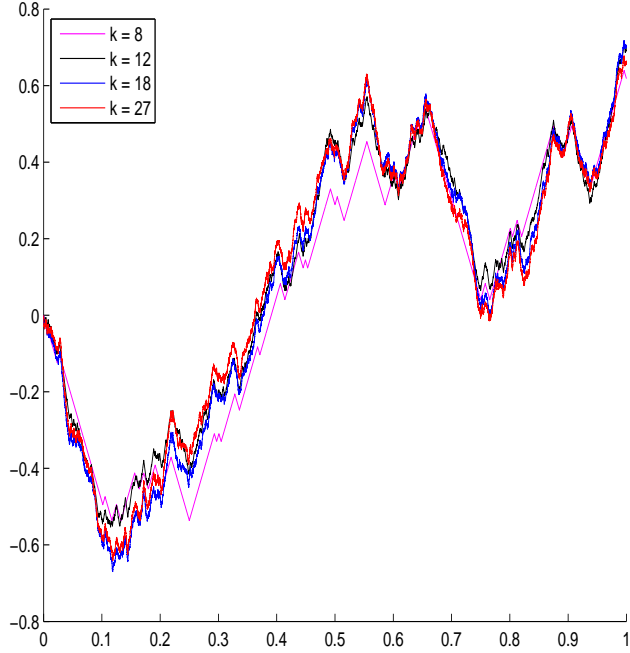}}
\vskip -.5cm
\caption{$B_k$ for $k=8,\ 12,\ 18,\ 27$ in the case $b=2$ and $H=0.7$. Strong convergence.}
\label{H=.7}
\end{center}
\end{figure}
\end{center}

\begin{center}
\begin{figure}
\begin{center}
{\includegraphics[width=10cm,height = 10cm]{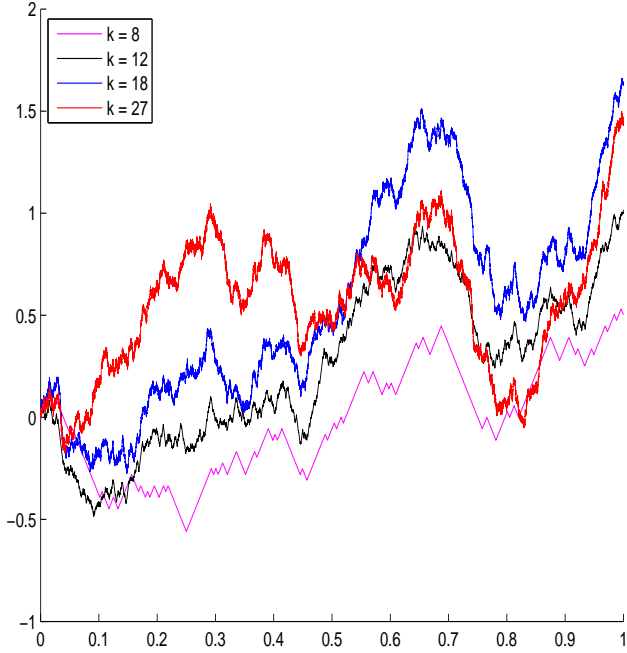}}
\vskip -1cm
\caption{$B_k/\sigma \sqrt{k}$ for $k=8,\ 12,\ 18,\ 27$ in the case  $b=2$ and $H=0.5$: Convergence in distribution to the Wierner Brownian motion.}
\label{H=.5}
\end{center}
\end{figure}
\end{center}
%

\begin{center}
\begin{figure}
\begin{center}
{\includegraphics[width=10cm,height = 10cm]{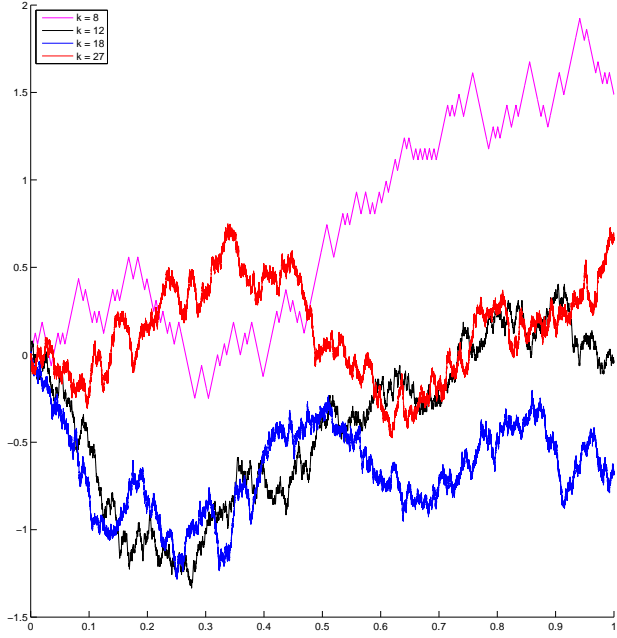}}
 \caption{$B_k/\sigma b^{k(1/2-H)}$ for $k=8,\ 12,\ 18,\ 27$ in the case $b=2$ and $H=-2$: Convergence in distribution to the Wierner Brownian motion.}
\label{H=-2}
\end{center}
\end{figure}
\end{center}

\section{Proof of Theorem~\ref{th2}}\label{sec2}

\label{strong}

\begin{lemma}
\label{lemmoments} The martingale $\big (Z_n=
B_{n}(1)\big )_{n\ge1}$ converges almost surely and in  $L^{q}$ norm
for all $q\ge1$.
\end{lemma}

\begin{proof}
For every integer $q\ge 1$, raising (\ref{fonc2}) to the power $q$ yields
\begin{equation}  \label{Mq}
\mathbb{E}(Z_{n+1}^q)=2^{1-qH}\mathbb{E}(\epsilon^q)\mathbb{E}
(Z_{n}^q)+2^{-Hq}\sum_{k=1}^{q-1}\binom{q}{k}\mathbb{E}(\epsilon^k)\mathbb{E}(\epsilon^{q-k})\mathbb{E}(Z_n^k)\mathbb{E}(Z_n^{q-k}).
\end{equation}
Moreover, since $H>1/2$ we have $0<2^{1-qH}\mathbb{E}(\epsilon^q)<1$ for all integers $q\ge 2$ ($\mathbb{E}(\epsilon^q)$ is equal to $2^{H-1}$ if $q$ is odd and 1 otherwise). Consequently, since $\mathbb{E}(Z_n)=1$ for all $n\ge 1$, induction on $q\in\mathbb{N}^*$ using (\ref{Mq}) shows that
the sequence $\mathbb{E}(Z_n^q)$ converges as $n$ tends to $\infty$ for every integer $q\ge 1$. This
implies that the martingale $(Z_n)_{n\ge 1}$ is bounded in $L^q$
norm for all $q\ge 1$, hence the result.
\end{proof}

\begin{lemma}\label{uniform}
Let $\alpha\in (0,H)$. With probability 1, there exists an integer $p_0\ge 1$ such that
$$
\forall\ p\ge p_0,\ \sup_{0\le k\le 2^{p}-1}\sup_{n\ge 1}\big | B_{n}\big ((k+1)2^{-p}\big )- B_{n}(k2^{-p})\big |\le 2^{-p\alpha}.
$$
\end{lemma}

\begin{proof}

For every $p\ge 1$ and $0\le k\le 2^p-1$, by construction the sequence\\ $\left (\Delta B_n(p,k)= B_{n}\big ((k+1)2^{-p}\big )- B_{n}(k2^{-p})\right )_{n\ge 1}$ is a martingale, so Doob's inequality yields for every $q>1$ a constant $C_q>0$ such that
$$
\mathbb{E}\left (\sup_{n\ge 1}\big |\Delta B_n(p,k)\big |^q\right )\le C_q\sup_{n\ge 1}\mathbb{E}\left (\big |\Delta B_n(p,k)\big |^q\right ).
$$
On the one hand --- always by construction --- if $n\le p$, then
$\mathbb{E}\left (\big |\Delta B_n(p,k)\big |^q\right
)=2^{-qn(H-1)}2^{-qp}\le 2^{-qpH}$ . On the other hand,
(\ref{self-sim}) and Lemma~\ref{lemmoments} together yield a
constant $C'_q\ge 1$ such that $\mathbb{E}\left (\big |\Delta
B_n(p,k)\big |^q\right )\le C'_q 2^{-qpH}$ if $n>p$. Consequently,
for all $p\ge 1$,
$$
\mathbb{E}\left (\sup_{n\ge 1}\big |\Delta B_n(p,k)\big |^q\right )\le C_qC'_q2^{-qpH}.
$$
For $q>(H-\alpha)^{-1}$, the previous inequality implies
$$
\sum_{p\ge 1}\mathbb{P}\left (\exists \ 0\le k<2^p: \sup_{n\ge 1}\big |\Delta B_n(p,k)\big |>2^{-p\alpha}\right )<\infty.
$$
We conclude thanks to the Borel-Cantelli lemma.
\end{proof}
For $w\in T$ we define $Z(w)=\lim_{n\to\infty}Z_n(w)$, and we denote $Z$ by $Z(\emptyset)$.
\begin{lemma}
\label{fourier} Let $\varphi$ stand for the characteristic function of
$Z$. There exists $\rho\in(0,1)$ such that $\varphi
(t)=O\big (
\rho^{|t|^{1/H}}\big )\ \ (|t|\to\infty)$. Consequently, the probability
distribution of $Z$ possesses an infinitely differentiable bounded density, and $\mathbb{E}(|Z|^{-\gamma})<\infty$ for all $\gamma\in (0,1)$.
\end{lemma}

\begin{proof}
The case $H=1$ is obvious. Suppose that $H\in (1/2,1)$. The probability distribution of
$Z$ cannot be a Dirac mass, because $\mathbb{E}(Z)=1$ and
\begin{equation}\label{FONC}
Z=2^{-H}\epsilon (0)Z(0)+2^{-H}\epsilon (1)Z(1),
\end{equation}
with the same independence and
 equidistribution properties as in (\ref{fonc2}). So there
exists $\alpha>0$ and $\gamma<1$ such that $\sup_{t, |t|\in
[\alpha,2^{H}\alpha]}|\varphi(t)|\le\gamma$. Now, using the fact that
\[
\varphi(t)=\left[ p_{H}^{+}\varphi\big (2^{-H}t\big )+p_{H}^{-}
\varphi\big (-2^{-H}t\big )\right] ^{2},
\]
we obtain by induction that $\displaystyle \sup_{t,\ |t|\in[2^{kH}\alpha,2^{(k+1)H] }\alpha]}|\varphi(t)|\le\gamma
^{2^{k}}\quad(\forall\ k\ge0)$. Since $|t|^{1/H}\le 2\alpha^{1/H} 2^{k}$ for $|t|\in[2^{kH}\alpha
,2^{(k+1)H}\alpha]$, the conclusion follows with $\rho=\gamma^{1/
2\alpha^{1/H}}$.

The rate of decay of $\varphi$ at $\infty$ yields the conclusion regarding the
probability distribution of $Z$ and the moments of $|Z|^{-1}$.
\end{proof}

\noindent {\it Proof of Theorem~\ref{th2}: the convergence
properties of $( B_{n})_{n\ge 1}$ and the global H\"older continuity
of the limit process.}

Let $\alpha\in (0,H)$. It follows from Lemma~\ref{uniform} that with probability 1, there exists $\delta>0$ and $C>0$ such that for all $(t,s)\in [0,1]^2$ such that $|t-s|\le \delta$ we have $\sup_{n\ge 1}| B_{n}(t)- B_{n}(s)|\le C |t-s|^{\alpha}$ (see for instance the proof of the Kolmogorov-Centsov theorem
in \cite{KS}). Since the sequence $( B_{n})_{n\ge 1}$ converges almost surely on the set of dyadic numbers of $[0,1]$ which is dense in $[0,1]$, this implies that, with probability 1, $( B_{n})_{n\ge 1}$ converges uniformly to a limit $ B $ which is $\alpha$-H\"older continuous. To see that the convergence holds in $L^q$ norm for all $q\ge 1$, it is enough to show that the sequence $\big (\mathbb{E}(\sup_{1\le p\le n}\| B_{p}\|_\infty^q)\big)_{n\ge 1}$ is bounded for all integer $q\ge 2$. We show that it is true for $q=2$ and leave the reader verify by induction that it is true for $q\ge 2$. For $n\ge 1$, define
\begin{equation*}
\widetilde Z_n=\sup_{1\le p\le n}\| B_{p}\|_\infty, \mbox{ and } \widetilde Z_{n}(k)=
\sup_{1\le p\le n}\| B_{p}(k)\|_\infty,\ k\in\{0,1\}.
\end{equation*}
Due to (\ref{self-sim}) we have for $n\ge 2$
\begin{equation*}
\widetilde Z_n\le \max\Big(2^{-H}\widetilde Z_{n-1}(0), 2^{-H}\widetilde Z_{n-1}(1)+\sup_{1\le p\le n}| B_{p}(1/2)|\Big ).
\end{equation*}
Thus, if we denote $\sup_{1\le p\le n}| B_{p}(1/2)|$ by $M_n$ we have
\begin{eqnarray*}
\mathbb{E}(\widetilde Z_n^2)\!\!\!&\le& \!\!\!\mathbb{E}\left (2^{-2H}\widetilde Z_{n-1}(0)^2+2^{-2H}\widetilde Z_{n-1}(1)^2+2\widetilde Z_{n-1}(1)M_n+M_n^2\right )
\\
\!\!\!&\le & \!\!\!2^{1-2H}\mathbb{E}(\widetilde Z_{n-1}^2)+2\mathbb{E}(\widetilde Z_{n-1}^2)^{1/2}
\left\|M_n\right \|_2+\| M_n\|_2^2.
\end{eqnarray*}
Lemma~\ref{lemmoments} shows that $(B_p(1/2))_{p\ge 1}$ is a
martingale bounded in $L^2$ norm, so $(\left\| M_n\right \|_2)_{n\ge
1}$ is bounded. Consequently, there exists $C>0$ such that
\begin{equation}  \label{infty}
\forall\ n\ge 1,\ \mathbb{E}(\widetilde Z_n^2)\le f\big (\mathbb{E}(\widetilde Z_{n-1}^2)\big ),\ \mbox{with } f(x)=
2^{1-2H}x+C\sqrt {x}+C.
\end{equation}
Since $2^{1-2H}<1$, there exists $x_0> 0$ such that $f(x)<x$ for all $x>x_0$
. This fact together with (\ref{infty}) yields $\mathbb{E}(\widetilde Z_n^2)\le
\max \left (x_0, f\big (\mathbb{E}(\widetilde Z_{1}^2)\big )\right )$ for all~$n\ge 2$.

\medskip

\noindent {\it Proof of Theorem~\ref{th2}: the properties 1., 2. and
3.}

1. This is an immediate consequence of (\ref{self-sim}).

2. The global H\"older regularity property has already been
established. To obtain the pointwise H\"older exponent we use an
approach similar to that used for the Brownian motion in \cite{Erdos} (see also \cite{KS}).

Fix $\varepsilon>0$ and let  $\mathcal{O}$ be the set of points 
$\omega\in \Omega$ such that $B_n$ converges uniforlmy as  $n\to\infty$ and the limit $ B $
possesses points at which the pointwise H\"older
exponent is at least $H+\varepsilon$. We show that $\mathcal{O}$ is included in a set of null
probability.

We fix an integer $K>4/\varepsilon$ and denote by $n_{K}$ the smallest integer
$n$ such that $K2^{-n}\le1$. For $t\in[0,1]$ and $n\ge n_{K}$ , consider
$S^{K}_{n}(t)$ a subset of $[0,1]$ consisting of $K+1$ consecutive dyadic numbers
of generation $n$ such that $t\in[\min\, S^{K}_{n}(t),\max\, S^{K}_{n}(t)]$.
Also denote by $\boldsymbol{S}^{K}_{n}(t)$ the set of $K$ consecutive dyadic
intervals delimited by the elements of $S^{K}_{n}(t)$. If the pointwise
H\"older exponent at $t$ is larger than or equal to $H+\varepsilon$ then for
$n$ large enough we have necessarily $\sup_{s\in S^{K}_{n}(t)}|B
(s)- B (t)|\le(K2^{-n})^{H+\varepsilon/2}$, so that $\sup
_{I\in\boldsymbol{S}^{K}_{n}(t)} |\Delta  B (I)|\le2 (K2^{-n}
)^{H+\varepsilon/2}$, where $\Delta  B (I)$ stands for the increment of $ B $ over $I$.

Now let $\boldsymbol{S}^{K}_{n}$ be the set consisting of all $K$-uple of
consecutive dyadic intervals of generation $n$, and if $S\in\boldsymbol{S}
^{K}_{n}$, denote the event $\big \{ \sup_{I\in S} |\Delta B
(I)|\le2 (K2^{-n})^{H+\varepsilon/2}\big \}$ by $E_{S}$. The previous
lines show that
\[
\mathcal{O}\subset\mathcal{O}^{\prime}=\bigcap_{n\ge n_{K}}\bigcup_{p\ge
n}\bigcup_{S\in\boldsymbol{S}^{K}_{p}}E_{S}.
\]
By construction, if $S\in\boldsymbol{S}^{K}_{p}$, $\big (|\Delta  B (I)|\big)_{I\in S}$ is equal to $( 2^{-pH}|Y_{I}|)_{I\in S}$, where the $K$
random variables $Y_{I}$ are mutually independent and identically distributed
with $ B (1)$. Consequently, $\mathbb{P}(E_{S})$ depends only on
$K$ and $p$ and
\begin{eqnarray*}
\mathbb{P}(E_{S})&\le&\left[  \mathbb{P}(| B (1)|\le2 K^H(K2^{-p}
)^{\varepsilon/2})\right] ^{K}\\
&\le& (2K^H)^{K/2} K^{K \varepsilon/4} 2^{-p K
\varepsilon/4}\left[ \mathbb{E}(| B (1)|^{-1/2})\right] ^{K},
\end{eqnarray*}
where $\mathbb{E}(| B (1)|^{-1/2})<\infty$ due to Lemma~\ref{fourier}. Since the cardinality of $\boldsymbol{S}^{K}_{p}$ is less than $2^{p}$, this
yields $\mathbb{P}\big(\bigcup_{S\in\boldsymbol{S}^{K}_{p}}E_{S}\big )=O(
2^{p} 2^{-p K \varepsilon/4})$. Our choice for
$K$ implies that the series $\sum_{p\ge n_K}\mathbb{P}\big(\bigcup_{S\in
\boldsymbol{S} ^{K}_{p}}E_{S}\big )$ converges, hence $\mathbb{P}(\mathcal{O}^{\prime})=0$.

\medskip

3. Let us introduce additional notations. If
$w\in\Sigma^{*}$ and $J=I_w$ then we define
$\boldsymbol{\epsilon}(J):=\boldsymbol{\epsilon}(w)=\prod_{k=1}^{|w|}\epsilon(w_1\cdots w_k)$. We denote by $\Gamma$ the graph $\left\{ \big (t,B(t)\big ):
t\in[0,1]\right\} $ of $B$. We recall that the Hausdorff dimension of a subset of $\mathbb{R}^2$ is always smaller than of equal to its box dimension.

\smallskip

At first, since $B$ is $\alpha$-H\"older continuous for all $\alpha<H$, $2-H$ is an upper bound for the box dimension of $\Gamma$ (see \cite{Falc} Ch. 11).

To find the sharp lower bound $2-H$ for the Hausdorff dimension of $\Gamma$ we show that, with probability 1, the
measure on this graph obtained as the image of the Lebesgue measure restricted
to $[0,1]$ by the mapping $t\mapsto\big (t,B(t)\big )$ has a finite energy
with respect to the Riesz Kernel $u\in\mathbb{R} ^{2}\setminus\{0\}\mapsto
\|u\|^{-\gamma}$ for all $\gamma<2-H$ (see \cite{Falc} Ch. 4.3 and 11 for 
details about this kind of approach). This property holds if we show that for all $\gamma<2-H$ we have
\[
\int_{[0,1]^{2}}\mathbb{E}\left( \big (|t-s|^{2}+|B
(t)-B(s)|^{2}\big )^{-\gamma/2}\right) \ \text{d}t\text{d}s <\infty.
\]
If $I$ is a closed subinterval of $[0,1]$, we denote by $\mathcal{G}(I)$ the
set of closed dyadic intervals of maximal length included in $I$, and then
$m_{I}=\min\bigcup_{J\in\mathcal{G}(I)}J$ and $M_{I} =\max\bigcup
_{J\in\mathcal{G }(I)}J$.

Let $0<s<t<1$ be two non dyadic numbers. We define two sequences
$(s_{p})_{p\ge0}$ and $(t_{p})_{p\ge0}$ as follows. Let $s_{0}=m_{[s,t]}$ and
$t_{0}=M_{[s,t]}$. Then let define inductively $(s_{p})_{p\ge1}$ and
$(t_{p})_{p\ge1}$ as follows: $s_{p}=m_{[s,s_{p-1}]}$ and $t_{p}
=M_{[t_{p-1},t]}$. Let us denote by $\mathcal{C}$ the collection of intervals
consisting of $[s_{0},t_{0}]$ and all the intervals $[s_{p},s_{p-1}]$ and
$[t_{p-1},t_{p}]$, $p\ge1$. Every interval $I\in\mathcal{C}$ is the union of
at most two intervals of the same generation $n_{I}$, the elements of
$\mathcal{G}(I)$, and
$$
\Delta  B (I)=\sum_{J\in\mathcal{G}(I)}\Delta  B (J)=\sum_{J\in\mathcal{G}(I)}
\boldsymbol{\epsilon} (J)2^{-n_{I}H}Y_{J},$$ where $ \Delta  B (J)$ and $Y(J)$ have been introduced in the discussion regarding the pointwise exponents. By construction, we have $\min_{I\in\mathcal{C} }n_{I}=n_{[s_{0},t_{0}]}$ and
$(t-s)/3\le 2^{-n_{[s_{0},t_{0}]}}\le(t-s)$. Also, all the random variables $Y_{J}$ are mutually independent and
independent of $\mathcal{T}_{\mathcal{C}}=\sigma(\boldsymbol{\epsilon}
(J):J\in\mathcal{G}(I),\ I\in\mathcal{C})$. Now, we write
$$
B(t)-B(s)=2^{-n_{[s_{0},t_{0}]}H}\Big( \sum_{J\in\mathcal{G}
([s_{0},t_{0}])} \boldsymbol{\epsilon} (J)Y_{J}+ Z(s,s_{0})+Z(t_{0},t)\Big)
,
$$
where
$$
\begin{cases}
\displaystyle Z(s,s_{0})=\lim_{p\to\infty} \sum_{0\le k\le p}2^{(n_{[s_{0}
,t_{0}]}-n_{[s_{k+1},s_{k}]})H} \sum_{J\in\mathcal{G} ([s_{k+1},s_{k}])}
\boldsymbol{\epsilon} (J)Y_{J}\\
\displaystyle Z(t_{0},t)=\lim_{p\to\infty} \sum_{0\le k\le p}2^{(n_{[s_{0}
,t_{0}]}-n_{[t_{k}, t_{k+1}]})H} \sum_{J\in\mathcal{G}([t_{k}, t_{k+1}])}
\boldsymbol{\epsilon} (J)Y_{J}.
\end{cases}
$$
Let $\mathcal{Z}(t,s)=\sum_{J\in\mathcal{G}([s_{0},t_{0}])}
\boldsymbol{\epsilon} (J)Y_{J}+ Z(s,s_{0})+Z(t_{0},t)$ and fix $J_{0}%
\in\mathcal{G}([s_{0},t_{0}])$. Conditionally on $\mathcal{T}_{\mathcal{C}}$,
$\mathcal{Z}(t,s)$ is the sum of $\pm Y(J_{0})$ plus a random variable $U$
independent of $Y(J_{0})$. Consequently, the probability distribution of
$\mathcal{Z}(t,s)$ conditionally on $\mathcal{T}_{\mathcal{C}}$ possesses a
density $f_{t,s}$ and $\|\widehat{f_{t,s}}\|_{L^{1}}\le\|\varphi\|_{L^{1}}$,
where $\varphi$ is the characteristic function of $Y(J_{0})$ studied in
Lemma~\ref{fourier}.

Thus, for $\gamma<2-H$ we have%

\begin{align*}
\mathbb{E}\left( \big (|t-s|^{2}+|B
(t)-B(s)|^{2}\big )^{-\gamma/2}|\mathcal{T }_{\mathcal{C}}\right)  & = \int_{\mathbb{R}}\frac{f_{t,s}%
(u)}{\big ( |t-s|^{2}+2^{-2n_{[s_{0},t_{0}]}H}u^{2}\big) ^{\gamma/2}}\, \text{d}u\\
& \le \int_{\mathbb{R}}\frac{f_{t,s}(u)}{\big ( |t-s|^{2}+3^{-2H}%
(t-s)^{2H}u^{2}\big )^{\gamma/2}}\, \text{d}u\\
& =|t-s|^{1-H-\gamma} \int_{\mathbb{R}}\frac{f_{t,s}(|t-s|^{1-H}v)}{\big (
1+3^{-2H}v^{2}\big )^{\gamma/2}}\, \text{d}v.
\end{align*}

The function $f_{t,s}$ is bounded independently of $t,\ s$ and $\mathcal{T}%
_{\mathcal{C}}$ since it is bounded by $\|\widehat{f_{t,s}}\|_{L^{1}}$ and we
just saw that this number is bounded by $\|\varphi\|_{L^{1}}$. Thus,
\[
\mathbb{E}\left(  \big (|t-s|^{2}+|B
(t)-B(s)|^{2}\big )^{-\gamma/2}\right) \le\|\varphi\|_{L^{1}}|t-s|^{1-H-\gamma}\int_{\mathbb{R}}\frac
{\text{d}v}{\big (1+3^{-2H}v^{2}\big )^{\gamma/2}} .
\]
This yields the conclusion. Notice that the fact that the distribution of the increment of $B$ over $[0,1]$, namely $Z$, has a density plays a crucial role in this proof, as the same kind of property is a powerful tool in finding a lower bound for the Hausdorff dimension of the graphs of fractional Brownian motions, symmetric L\'evy processes of index $\alpha\in (1,2)$ and certain Weierstrass functions with random phases (see \cite{Falc,Hunt}).

\section{Proof of Theorem~\ref{newTCL}}\label{sec3}
The case $H=-\infty$ has been discussed in Remark \ref{rem1.3}. We fix $H\in (-\infty,1/2]$.
\begin{lemma}\label{ident}
\label{prop1} The sequence $(X_{n}(1))_{n\ge1}$ converges in law to the standard normal distribution as $n$
tends to $\infty$.
\end{lemma}

\begin{proof}
Let $u_0=\mathbb{E}(Z_{0}^{2})=1$. By definition, we have $u_0=1$. Let $\ell$ be the solution of $\ell=2^{1-2H}\ell +\frac{1}{2}$ when $H<1/2$, i.e. $\ell=(2-2^{2-2H})^{-1}$. Taking successively the square and the expectation in (\ref{fonc2}) yields $\displaystyle \mathbb{E}(Z_{n}^{2}) =2^{1-2H}\mathbb{E}(Z_{n-1}^{2})+\frac{1}{2}$ for $n\ge 1$. Consequently, $\mathbb{E}(Z_{n}^{2})=\ell+2^{n(1-2H)}(u_0-\ell)$ if $H>1/2$ and $\mathbb{E}(Z_{n}^{2})=u_0+n/2$ if $H=1/2$. This yields $\mathbb{E}(Z_n^2)\sim \frac{2^{2-2H}-1}{2^{2-2H}-2}2^{n(1-2H)}=\sigma^2 2^{n(1-2H)}$ if $H<1/2$ and $\mathbb{E}(Z_n^2)\sim n/2=\sigma^2n$ if $H=1/2$.  This is why we consider the normalized processes~$X_{n}$.

For $n\ge 1$ and $q\ge 1$ let $M_n^{(q)}=\mathbb{E}(X_{n}(1)^q)$. We are going to prove by induction and by using (\ref{fonc2}) that
\begin{enumerate}
\item for every $p\ge0$ one has the property $(\mathcal{P}_{2p})$:
$M^{(2p)}=\lim_{n\to\infty} M_n^{(2p)}$ exists; moreover $M^{(2)}=1$;

\item for every $p\ge0$ one has the property $(\mathcal{P}_{2p+1})$:
$\lim_{n\to\infty} M_n^{(2p+1)}=0$;

\item the sequence $(M^{(2p)})_{p\ge 1}$ obeys the following induction relation valid
for $p\ge2$:
\begin{equation}\label{momnorm}
M^{(2p)}=\big (2^{p}-2\big )^{-1}\sum_{k=1}^{p-1}\binom{2p}{2k} M^{(2k)}M^{(2p-2k)}.
\end{equation}
\end{enumerate}

Suppose that these properties have been established. Then, 1. insures
that the probability distributions of the $X_n(1)$ form a tight
sequence. Moreover, it is easy to verify that a $\mathcal{N }(0,1)$
random variable $N$ has the property that its
moments of even orders satisfy the same relation as the
numbers $M_{2p}$, $p\ge1$, defined by $M_{2}=1$ and the induction
relation 3. To see this, write $N$ as the sum of two
independent $\mathcal{N}(0,2^{-1/2})$ random variables.
Consequently, since the law $\mathcal{N}(0,1)$ is characterized by
its moments, 1., 2. and 3. imply that $X_n(1)$ converges in law to
$\mathcal{N}(0,1)$.

\smallskip

Now we prove 1., 2., and 3.. By construction, we have $M_n^{(1)}\sim 1/(\mathbb{E}(Z_n^2))^{1/2}$ hence $\lim_{n\to\infty} M_n^{(1)}=0$, as well as $\lim_{n\to\infty}M_n^{(2)}=1$. Consequently, $(\mathcal{P}_1)$ and $(\mathcal{P}_2)$~hold.

\smallskip

Let $q$ be an integer $\ge 3$. Raising
(\ref{foncnorm}) to the power $q$ yields

\begin{equation}  \label{H0}
M_{n+1}^{(q)}= \displaystyle r_n^{q}\left
(2^{1-q/2}\mathbb{E}(\epsilon^q)\mathbb{E}
(Z_{n}^q)+2^{-q/2}S(q,n)\right ),
\end{equation}
where $r_n=\displaystyle\sqrt{\frac{n
}{n +1}}$ if $H=1/2$ and $r_n=1$ otherwise, and
$$
\displaystyle S(q,n)=\sum_{k=1}^{q-1}\binom{q}{k}\mathbb{E}(\epsilon^k)\mathbb{E}(\epsilon^{q-k}) M_n^{(k)}M_n^{(q-k)}.
$$
Since $\mathbb{E}(\epsilon_0^q)=2^{H-1}$ or $1$ according to $q$ is
odd or even, (\ref{H0}) yields

\begin{equation}  \label{H}
M_{n+1}^{(q)}=
\begin{cases}
\displaystyle r_n^{q}\left
(2^{H-q/2}M_n^{(q)}+2^{-q/2}S(q,n)\right )
& \mbox{if $q$ is odd}, \\
\displaystyle r_n^{q}\left
(2^{1-q/2}M_n^{(q)}+2^{-q/2}S(q,n)\right )
& \mbox{if $q$ is even}
\end{cases}.
\end{equation}
Let us show by induction that $\big
((\mathcal{P}_{2p-1}),(\mathcal{P}_{2p}) \big )$ holds for $p\ge 1$, as well as (\ref{momnorm}).

We have already shown that $\big
((\mathcal{P}_{1}),(\mathcal{P}_{2})\big )$ holds. Suppose that
$\big ((\mathcal{P}_{2k-1}),(\mathcal{P}_{2k})\big )$ holds for
$1\le k\le p-1$, with $p\ge 2$. In particular, $M_n^{(k)}$ tends to
0 as $n$ tend to $\infty$ if $k$ is an odd integer belonging to $[1,
2p-3]$. Consequently, $S(2p-1,n)$ tends to 0 as $n$ tends to
$\infty$; indeed, for each integer $k$ between $1$
and $2p-1$, either $k$ or $2p-1-k$ is an odd number. The sequence
$(r_n)_{n\ge 1}$ being bounded, it follows from this property and
(\ref{H}) that $\displaystyle 
M_{n+1}^{(2p-1)}=r_n^{2p-1}2^{H+1/2-p}M_n^{(2p-1)}+o(1)$ as $n\to\infty$. Since $r_n^{2p-1}2^{H+1/2-p}\le 2^{1-p}<1$, this yields $\lim_{n\to\infty}M_n^{(2p-1)}=0$, that is to say $(\mathcal{P}_{2p-1})$.

Also, our induction's assumption implies that in the right hand side of $
M_{n+1}^{(2p)}$, the term $S(2p,n)$ tends to
$L=\sum_{k=1}^{p-1}\binom{2p}{2k}M^{(2k)}M^{(2p-2k)}$
as $n$ tends to $\infty$. Define $
L'=(2^p-2)^{-1} L$. By using (\ref{H}) we deduce from the previous
lines that
\begin{eqnarray*}
M_{n+1}^{(2p)}=
\begin{cases}
r_n^{2p}2^{1-p}M_n^{(2p)}+2^{-p}L+o(1)&\text{if } H=1/2\\
2^{1-p}M_n^{(2p)}+2^{-p}L+o(1)&\text{if } H<1/2
\end{cases}.
\end{eqnarray*}
Since $r_n\to 1$ as $n\to\infty$ when $H=1/2$ and $L'=
2^{1-p}L'+2^{-p} L$ we obtain
$$
M_{n+1}^{(2p)}-L'=
\begin{cases}
r_n^{2p}2^{1-p}(M_n^{(2p)}-L')+o(1)&\text{if } H=1/2\\
2^{1-p}(M_n^{(2p)}-L')+o(1)&\text{if } H<1/2
\end{cases}.
$$
This yields both $(\mathcal{P}_{2p})$ and (\ref{momnorm}) since  $r_n\le 1$ and $2^{1-p}<1$.
\end{proof}

\begin{lemma}\label{tight}
The laws of the random continuous functions $X_n$, $n\ge 1$, form a tight family in the set of probability measures on $\mathcal{C}([0,1])$.
\end{lemma}
\begin{proof}
By Theorem 7.3 of \cite{Bil}, since
$X_n(0)=0$ almost surely for all $n\ge1$, it is enough to show
that for each positive $\varepsilon$
\begin{equation}
\label{tighness}\lim_{\delta\to0}\limsup_{n\to\infty} \mathbb{P}
\big (\omega(X _{n},\delta)\ge\varepsilon\big )=0,
\end{equation}
where $\omega(X_n,\cdot)$ stands for the modulus of continuity of $X_n$.

We leave the reader to check the following simple properties for $p,n\ge 1$ and $w\in\{0,1\}^p$: If $n>p$ then

\begin{equation}\label{self-sim2}
X_{n}(t_w+2^{-p})-X_{n}(t_w)=
\boldsymbol{\epsilon}(w)2^{-p/2}\cdot
\begin{cases}
X_{n-p}(w)(1)&\mbox{if $H<1/2$}\\
\displaystyle \sqrt{\frac{n-p}{n}}X_{n-p}(w)(1)&\mbox{if $H= 1/2$}
\end{cases}
\end{equation}
and if $1\le n\le p$ then
 \begin{equation}\label{self-sim3}
|X_n(t_w+2^{-p})-X_{n}(t_w)|\le
\begin{cases}
\displaystyle 2^{-p/2}/\sigma &\mbox{if $H<1/2$}\\
\displaystyle 2^{-p/2}/\sigma \sqrt{n}&\mbox{if $H=1/2$}\\
\end{cases}.
\end{equation}
Moreover, the proof of Lemma~\ref{prop1} shows that $\sup_{n\ge 1}\mathbb{E}(X_n(1)^{2K})<\infty$ for every integer $K\ge 1$. Consequently, it follows from (\ref{self-sim2}) and (\ref{self-sim3}) that there exists a family $\{V_{n,p,k}\}_{n,p\ge 1,0\le k\le 2^p-1}$ of positive random variables such that 
$$
\left | X_{n}\big ((k+1)2^{-p}\big)-
X_{n}\big (k2^{-p}\big ) \right|\le 2^{-p/2}V_{n,p,k},
$$
and for any integer $K\ge 1$, $C_K=\sup_{\substack{n,p\ge 1\\0\le k\le 2^p-1}}\mathbb{E}(V_{n,p,k}^{2K})<\infty$. The end of the proof is then standard.

Fix $\alpha\in(0,1/2)$ and $K$ a positive integer such that
$2K(1/2-\alpha)>1$. Define $\rho_p=C_K2^{p(1+2K(\alpha-1/2))}$ and $R_{p}=\sum_{j\ge p}\rho_{j}$ for $p\ge 1$. For all $n,p\ge 1$, our control of the moments of the dyadic increments of $X_n$ yields, using Markov inequalities, $
 \mathbb{P}\Big(\bigcup_{0\le k< 2^{p}}\ \left\{\left| X_{n}
\big ( (k+1)2^{-p}\big)- X_{n}\big (k2^{-p}\big )\right|
>2^{-p\alpha}\right \}\Big) \le \rho_{p}$.

Thus, $
\inf_{n\ge1} \mathbb{P}(E^n_p)\ge 1-R_p$ for all $p\ge1$, where
$$
E^n_p=\left\{ \forall\ j\ge p,\ \forall\ 0\le k< 2^{-j},\ \left|
X_{n}\big ((k+1)2^{-j}\big)- X_{n}\big (k2^{-j}\big )
\right| \le 2^{-j\alpha}\right\}.
$$
Also, on $E^{n}_{p}$ we have
$\displaystyle\omega\big (X_{n}
,2^{-p}\big )\le 2^{1-p\alpha}/(1-2^{-\alpha}).$
This yields
$$
\inf_{n\ge 1} \mathbb{P}\left( \omega\big (X_{n},2^{-p}\big )\le
2^{1-p\alpha}/(1-2^{-\alpha})\right) \ge
\inf_{n\ge1} \mathbb{P}(E^{n}_{p})\ge1-R_{p}.
$$
Since $\lim_{p\to\infty}R_{p}=0$, the previous inequality gives
(\ref{tighness}).
\end{proof}

\noindent
{\it Proof of Theorem~\ref{newTCL}.} Since for all $p\ge 1$ the random sequences $(X_{n}(w))_{n\ge 1}$, $w\in\{0,1\}^p$, are mutually independent, it follows from (\ref{self-sim2}) and Lemma~\ref{prop1} that for all $p\ge 1$,
 the sequence of vectors $V_{n}(p)=\big (X_{p+n}(t_w+2^{-p})-X_{p+n}(t_w)\big )_{w\in\{0,1\}^p}$ converges in law, as $n$ tends to $\infty$, to the distribution of the increments of the standard Brownian motion on the dyadic subintervals of $[0,1]$ of generation $p$. This is seen by taking the limit as $n$ tends to $\infty$ of  the characteristic function of $V_{n}(p)$ conditionally on $\sigma\big ( \boldsymbol{\epsilon}(w),\ w\in\{0,1\}^p\big )$ and then by using the fact that $\boldsymbol{\epsilon}(w)^2=1$. Consequently, the only possible weak limit of a subsequence of $(X_n)_{n\ge 1}$ is the standard Brownian motion. Then Lemma~\ref{tight} yields the desired conclusion.

\section{Proof of Theorem~\ref{newTCL1}}\label{sec4}
Theorem~\ref{newTCL1} follows from the next proposition. For $H\in (1/2,1)$ and $w\in T$ we denote $B_H(w)/\sigma_H$ by $\widetilde B_H(w)$ ($\widetilde B_H(\emptyset)=B_H/\sigma_H$ is denoted $\widetilde B_H$).
\begin{proposition}\label{ident2}
Let $(H_m)_{m\ge 1}$ be a $(1/2,1)$-valued sequence converging to $1/2$ as $n\to\infty$.

\begin{enumerate}
\item
The sequence $(\widetilde B_{H_m}(1))_{m\ge1}$ converges in law to the standard normal distribution as $m$
tends to $\infty$.

\item The laws of the random continuous functions $\widetilde B_{H_m}$, $m\ge 1$, form a tight family in the set of probability measures on $\mathcal{C}([0,1])$.

\item For every $p\ge 1$, the sequence of vectors $\widetilde V_{m}(p)=\big (\widetilde B_{H_m}(t_w+2^{-p})-\widetilde B_{H_m}(t_w)\big )_{w\in\{0,1\}^p}$ converges in law, as $m$ tends to $\infty$, to the distribution of the increments of the standard Brownian motion on the dyadic subintervals of $[0,1]$ of generation $p$.
\end{enumerate}
\end{proposition}

\begin{proof}
1. The proof is close to that of Lemma~\ref{ident}, but the differences deserve to be made explicit.

For every $q,m\ge 1$, let us denote $\mathbb{E}(\widetilde B_{H_m}(1)^q)$ by $\widetilde M^{(q)}_{m}$. Since $H=H_m>1/2$ and by definition $\widetilde B_{H_m}(1)=\sqrt{2-2^{2-2H_m}}B(1)=\sqrt{2-2^{2-2H_m}}Z$, taking the limit in (\ref{Mq}) as $n\to\infty$ thanks to Lemma~\ref{lemmoments} and using the fact that $\mathbb{E}(\epsilon_0^q)=2^{H-1}$ or $1$ according to $q$ is
odd or even, we obtain
\begin{equation}  \label{H'}
\widetilde M_{m}^{(q)}=
\begin{cases}
\displaystyle
2^{-(q-1)H_m}\widetilde M_{m}^{(q)}+2^{-qH_m}\widetilde S(q,m)
& \mbox{if $q$ is odd}, \\
\displaystyle 2^{1-qH_m}\widetilde M_{m}^{(q)}+2^{-qH_m}\widetilde S(q,m)
& \mbox{if $q$ is even}
\end{cases},
\end{equation}
where
$
\displaystyle \widetilde S(q,m)=\sum_{k=1}^{q-1}\binom{q}{k}\mathbb{E}(\epsilon^k)\mathbb{E}(\epsilon^{q-k}) \widetilde M_{m}^{(k)}\widetilde M_{m}^{(q-k)}.
$
Now we prove by induction that
\begin{enumerate}
\item for every $p\ge0$ one has the property $(\mathcal{P}_{2p})$:
$\widetilde M^{(2p)}=\lim_{m\to\infty} \widetilde M^{(2p)}_m$ exists. Moreover $\widetilde M^{(2)}=1$;

\item for every $p\ge0$ one has the property $(\mathcal{P}_{2p+1})$:
$\lim_{m\to\infty}\widetilde M^{(2p+1)}_m=0$;

\item the sequence $(\widetilde M^{(2p)})_{p\ge 1}$ obeys the same induction relation (\ref{momnorm}) as the sequence $(M^{(2p)})_{p\ge 1}$ defined in the proof of Lemma~\ref{ident}.
\end{enumerate}
The conclusion is then the same as in the proof of Lemma~\ref{ident}.

\smallskip

To prove that $(\mathcal{P}_{1})$ and $(\mathcal{P}_{2})$ hold we first recall that $H$ being fixed, we have seen in the proof of Lemma~\ref{ident} that $\displaystyle \mathbb{E}(Z_{n}^{2}) =2^{1-2H}\mathbb{E}(Z_{n-1}^{2})+\frac{1}{2}$. For $H>1/2$ this yields $\mathbb{E}(Z^2)=\lim_{n\to\infty}\mathbb{E}(Z_n^2)=(2-2^{2-2H})^{-1}$. Consequently, $\mathbb{E}(\widetilde B_H(1))=\sqrt{2-2^{2-2H}}\mathbb{E}(B(1))=\sqrt{2-2^{2-2H}}$ tends to $0$ as $H\searrow 1/2$ and $\mathbb{E}(\widetilde B_H(1)^2)=1$.

\smallskip

Suppose that $\big ((\mathcal{P}_{2k-1}),(\mathcal{P}_{2k})\big )$ holds for
$1\le k\le p-1$, with $p\ge 2$. The same approach as in the proof of Lemma~\ref{ident} implies that in (\ref{H'}), the term $2^{-(2p-1)H_m}\widetilde S(2p-1,m)$ in the right hand side of $\widetilde M_m^{(2p-1)}$ tends to 0 as $m$ tends to
$\infty$. This implies $\displaystyle \widetilde M_{m}^{(2p-1)}=2^{-(2p-2)H_m}\widetilde M_m^{(2p-1)}+o(1)$
as $m\to\infty$. Since $2^{-(2p-2)H_m}\le 2^{-(p-1)}<1$, this yields $\lim_{m\to\infty}M_m^{(2p-1)}=0$, that is to say $(\mathcal{P}_{2p-1})$. The induction's assumption also implies that in the right hand side of $\widetilde M_{m}^{(2p)}$, the term $\widetilde S(2p,m)$ tends to
$L=\sum_{k=1}^{p-1}\binom{2p}{2k}\widetilde M^{(2k)}\widetilde M^{(2p-2k)}$ as $m$ tends to $\infty$. Define $
L'=(2^p-2)^{-1} L$. By using (\ref{H'}) we deduce from the previous
lines that $M_{m}^{(2p)}=2^{1-2pH_m}M_m^{(2p)}+2^{-p}L+o(1)$ as $m\to\infty$. As $2^{1-2pH_m}$ tends to $2^{1-p}$ as $m\to\infty$, the definition of $L'$ implies $
M_{m}^{(2p)}-L'=2^{1-2pH_m}(M_m^{(2p)}-L')+o(1)
$ as $m\to\infty$. Since $2^{1-p}<1$ the last equality yields both $(\mathcal{P}_{2p})$ and (\ref{momnorm}) for $(\widetilde M^{(2p)})_{p\ge 1}$ instead of $(M^{(2p)})_{p\ge 1}$.

\smallskip

\noindent 2. If $H\in (1/2,1)$, $p\ge 1$ and $w\in \{0,1\}^p$, due to Theorem~\ref{th2}.1 we have
\begin{equation}\label{FONC1}
\widetilde B_H(t_w+2^{-p})-\widetilde B_H(t_w)=\boldsymbol{\epsilon}(w)2^{-pH}\widetilde B_H(w)(1).
\end{equation}
This implies $\left | \widetilde B_{H_m}(t_w+2^{-p})-\widetilde B_{H_m}(t_w) \right|\le 2^{-p/2} |\widetilde B_{H_m}(w)(1)|$. Moreover, the proof of 1. above shows that $C_K=\sup_{m\ge 1}\mathbb{E}(|\widetilde B_{H_m}(1)|^{2K})<\infty$ for every integer $K\ge 1$. We conclude as in the proof of Lemma~\ref{tight}.

\smallskip

\noindent 
3. Use (\ref{FONC1}) and the same arguments as in the proof of Theorem~\ref{newTCL} as well as the fact that $2^{-pH_m}$ tends to $2^{-p/2}$ as $m\to\infty$.
\end{proof}

\end{document}